\documentclass[12pt]{article}
\usepackage{amsmath,amsthm,amssymb,epsfig,graphics,subfigure, wrapfig, epsfig, color}
\input amssym.def
\input amssym.tex
\date{}

\newtheorem{theorem}{Theorem}
\newtheorem{corollary}{Corollary}
\newtheorem{remark}{Remark}

\begin{document}

\title{{\bf $L_p +L_q$ and $L_p \cap L_q$ are not isomorphic \newline for all $1 \leq p, q \leq \infty$, $p \neq q$} \\
{\bf \large $L_p + L_q$ et $L_p \cap L_q$ ne sont pas isomorphes pour tout \newline $1 \leq p, q \leq \infty $, $p \neq q$}}

\author {Sergey V. Astashkin$^{a, *}$ and Lech Maligranda$^{b}$}

\date{}

\maketitle

$^a$Department of Mathematics, Samara National Research University, Moskovskoye 

shosse 34, 443086, Samara, Russia

$^b${Department of Engineering Sciences and Mathematics, Lule\r{a} University of Technology, 

SE-971 87 Lule\r{a}, Sweden

\renewcommand{\thefootnote}{\fnsymbol{footnote}}

\noindent
\footnotetext[0]{{\it E-mail addresses:} {\tt astash56@mail.ru} (S. V. Astashkin), {\tt lech.maligranda@ltu.se} (L. Maligranda)}

\footnotetext[1]{Research was supported by the Ministry of Education and Science of the Russian
Federation, project 1.470.2016/1.4, and by the RFBR grant 17-01-00138.}

\begin{abstract}
\noindent {We prove that if $1 \leq p, q \leq \infty$, then the spaces $L_p +L_q$ and $L_p \cap L_q$ are 
isomorphic if and only if $p = q$. In particular, $L_2 +L_{\infty}$ and $L_2 \cap L_{\infty}$ are not isomorphic which 
is an answer to a question formulated in \cite{AM}.}
\end{abstract}

\hspace{7cm}{\bf R\'esum\'e}
\begin{quote}
{\small Nous prouvons que si $1 \leq p, q \leq \infty$, alors les espaces $L_p +L_q$  et $L_p \cap L_q$ 
sont isomorphes si et seulement si $p = q$. En particulier, $L_2 +L_{\infty}$ et $L_2 \cap L_{\infty}$ 
ne sont pas isomorphes, ce qui est une réponse \`a une question formul\'ee dans \cite{AM}.}
\end{quote}

\section{Preliminaries and main result}

Isomorphic classification of symmetric spaces is an important problem related to the study of symmetric structures 
in arbitrary Banach spaces (cf. \cite{JMST}). A number of very interesting and deep results of such a sort is 
proved in the seminal work of Johnson, Maurey, Schechtman and Tzafriri \cite{JMST}. In particular, in \cite{JMST} 
(see also \cite[Section 2.f]{LT}) it was shown that the space $L_2 \cap L_p$ for $2 \leq p < \infty$ (resp. $L_2+L_p$ 
for $1 < p \leq 2$) is isomorphic to $L_p$. A further investigation of various properties of \textit{separable} sums and 
intersections of $L_p$-spaces (i.e., with $p<\infty$) was continued by Dilworth in \cite{Di88} and \cite{Di90} and by 
Dilworth and Carothers in \cite{DC91}. In contrast to that, in the paper \cite{AM} we proved that \textit{nonseparable} 
spaces $L_p +L_{\infty}$ and $L_p \cap L_{\infty}$ for all $1\leq p<\infty$ and $p \neq 2$ are not isomorphic. 
The question for $p = 2$ was our motivation to continue this work. Here, we give a solution of this problem and, 
basing on results of \cite{JMST} and \cite{AM}, prove a more general theorem: $L_p +L_q$ and $L_p \cap L_q$ 
for all $1 \leq p, q \leq \infty$ are isomorphic if and only if $p = q$.
\vspace{2mm}
 
In this paper we use the standard notation from the theory of symmetric spaces (cf. \cite{BS}, \cite{KPS} and \cite{LT}). 
Let $L_p(0, \infty)$ be the usual Lebesgue space of $p$-integrable functions $x(t)$ equipped with the norm
\begin{equation*}
\| x \|_{L_p} = \left( \int_0^{\infty} |x(t)|^p dt \right)^{1/p} ~~ (1 \leq p < \infty)
\end{equation*}
and $\| x\|_{L_{\infty}} = {\rm ess \,sup}_{t > 0} |x(t)|$. For $1 \leq p, q \leq \infty$ the space $L_p +L_q$ consists of all 
sums of $p$-integrable and $q$-integrable measurable functions on $(0, \infty)$ with the norm defined by
\begin{equation*}
\| x \|_{L_p + L_q}: = \inf_{x(t) = u(t) + v(t), \,u \in L_p, v \in L_q} \left( \| u \|_{L_p} + \| v \|_{L_q} \right).
\end{equation*}
The space $L_p \cap L_q$ consists of all both $p$- and $q$-integrable functions on $(0, \infty)$ with the norm
\begin{equation*}
\| x \|_{L_p \cap L_q}: = \max \left \{ \| x \|_{L_p}, \| x \|_{L_q} \right \} =  \max \left \{\big(\int_0^{\infty} | x(t)|^p\, dt \big)^{1/p}, 
\big(\int_0^{\infty} | x(t)|^q\, dt\big)^{1/q} \right \}.
\end{equation*}
$L_p +L_q$ and $L_p \cap L_q$ for all $1 \leq p, q \leq \infty$ are symmetric Banach spaces (cf. \cite[p. 94]{KPS}). 
They are separable if and only if both $p$ and $q$ are finite (cf. \cite[p. 79]{KPS} for $p = 1$).

The norm in $L_p +L_q$ satisfies the following estimates
$$
\Big(\int_0^1 x^*(t)^p \, dt\Big)^{1/p}+\Big(\int_1^\infty x^*(t)^q \, dt\Big)^{1/q} \leq \| x \|_{L_p + L_q}\le$$
$$\leq C_{p,q} \Big(\Big(\int_0^1 x^*(t)^p \, dt\Big)^{1/p}+\Big(\int_1^\infty x^*(t)^q \, dt\Big)^{1/q}\Big) 
$$
if $1\le p<q<\infty$, and
\begin{equation*} \label{1}
\Big(\int_0^1 x^*(t)^p \, dt\Big)^{1/p} \leq \| x \|_{L_p + L_{\infty}} \leq C_p \Big(\int_0^1 x^*(t)^p \, dt\Big)^{1/p} 
\end{equation*}
if $1\le p<\infty$
(cf. \cite[p. 109]{BL}, \cite[Thm 4.1]{Ho70} and \cite[Example 1]{Ma84}). Here, $x^*(t)$ denotes the decreasing rearrangement 
of $| x(u)|$, that is, 
\begin{equation*}
 x^*(t) = \inf \{\tau>0 \colon m(\{u > 0 \colon |x(u)| > \tau\}) < t \} 
\end{equation*}
(if $E\subset\mathbb{R}$ is a measurable set, then $m(E)$ is its Lebesgue measure). Note that every 
measurable function and its decreasing rearrangement are equimeasurable, that is,
$$
m(\{u > 0 \colon |x(u)| > \tau\})=m(\{t > 0 \colon |x^*(t)| > \tau\})$$
for all $\tau>0$.
\vspace{2mm}

Now, we state the main result of this paper.

\begin{theorem} \label{Thm1}
For every $1 \leq p, q \leq \infty$ the spaces $L_p +L_q$ and $L_p \cap L_q$ are isomorphic if and only if $p = q$.
\end{theorem}

If $\{x_n\}_{n=1}^\infty$ is a sequence from a Banach space $X$, by $[x_n]$ we denote its closed linear span in $X$. As usual,
the Rademacher functions on $[0, 1]$ are defined as follows: $r_k(t) = {\rm sign}
(\sin 2^k \pi t), ~ k \in {\Bbb N}, t \in [0, 1]$.

\section{ $L_2 + L_{\infty}$ and $L_2 \cap L_{\infty}$ are not isomorphic}

Let $x$ be a measurable function on $(0,\infty)$ such that $m(\text{supp} \,x)\le 1$. Then, clearly, $x$ is equimeasurable 
with the function $x^*\chi_{[0,1]}$. Therefore, assuming that $x\in L_2$ (resp. $x\in L_\infty$), we have 
$x\in L_2+L_\infty$ and $\|x\|_{L_2+L_\infty}=\|x\|_{L_2}$ (resp. $x\in L_2\cap L_\infty$ and $\|x\|_{L_2\cap L_\infty}=\|x\|_{L_\infty}$).

\begin{theorem} \label{Thm2}
The spaces $L_2 + L_{\infty}$ and $L_2 \cap L_{\infty}$ are not isomorphic.
\end{theorem}

\begin{proof} On the contrary, assume that $T$ is an isomorphism of $L_2 + L_{\infty}$ onto $L_2 \cap L_{\infty}$. 

For every $n, k \in \mathbb N$ and $i = 1, 2, \ldots, 2^k$ we set
\begin{equation*}
\Delta_{k, i}^{n} = (n - 1 + \frac{i-1}{2^k}, n - 1 + \frac{i}{2^k}], \, u_{k, i}^{n}: = \chi_{\Delta_{k, i}^{n}}, \, v_{k, i}^{n}: = T(u_{k, i}^{n}).
\end{equation*}
Clearly, $\| u_{k, i}^{n}\|_{L_2+L_{\infty}} = 2^{-k/2}$. Therefore, if $x_{k, i}^{n} = 2^{k/2} u_{k, i}^{n}, y_{k, i}^{n} = 2^{k/2} v_{k, i}^{n}$, then
$\| x_{k, i}^{n}\|_{L_2+L_{\infty}} = 1$ and
\begin{equation} \label{1}
\| T^{-1}\|^{-1} \leq \| y_{k, i}^{n}\|_{L_2 \cap L_{\infty}} = \max (\| y_{k, i}^{n}\|_{L_2}, \| y_{k, i}^{n}\|_{L_{\infty}}) \leq \| T\|
\end{equation}
for all $n, k \in {\mathbb N}, i = 1, 2, \ldots, 2^k$.

At first, we suppose that for each $k \in \mathbb N$ there are $n_k \in \mathbb N$ and $1 \leq i_k \leq 2^k$ such that
\begin{equation} \label{2}
\| y_{k, i_k}^{n_k}\|_{L_2} \rightarrow 0 ~~ \text{as} ~~ k \rightarrow \infty.
\end{equation}
Denoting $\alpha_k:= x_{k, i_k}^{n_k}$ and $\beta_k:= y_{k, i_k}^{n_k}$, observe that $m(\bigcup_{k=1}^{\infty} \text{supp} \, \alpha_k) = 1$ 
and so the sequence $\{\alpha_k\}_{k=1}^{\infty}$ is isometrically equivalent in $L_2+L_{\infty}$ to the unit vector basis of $l_2$ and 
$[\alpha_k]$ is a complemented subspace of $L_2+L_{\infty}$. Then, since $\beta_k = T(\alpha_k), k = 1, 2, \ldots$, the sequence 
$\{\beta_k\}_{k=1}^{\infty}$ is also equivalent in $L_2 \cap L_{\infty}$ to the unit vector basis of $l_2$. Moreover, if $P$ is a bounded 
projection from $L_2+L_{\infty}$ onto $[\alpha_k]$, then the operator $Q:= TPT^{-1}$ is the bounded projection from $L_2 \cap L_{\infty}$ 
onto $[\beta_k]$. Thus, the subspace $[\beta_k]$ is complemented in $L_2 \cap L_{\infty}$.

Now, let $\varepsilon_k > 0, k =1, 2, \ldots$ and $\sum_{k=1}^{\infty} \varepsilon_k < \infty$ (the choice of these numbers will be 
specified a little bit later). Thanks to (\ref{2}), passing to a subsequence (and keeping the notation), we may assume that
\begin{equation*}
\| \beta_k \|_{L_2} < \varepsilon_k ~~\text{and} ~~ m\{s > 0: | \beta_k(s)| > \varepsilon_k \} < \varepsilon_k, k = 1, 2, \ldots
\end{equation*}
(clearly, this subsequence preserves the above properties of the sequence $\{\beta_k\}$). Hence, denoting
\begin{equation*}
A_k:= \{s > 0: | \beta_k(s)| > \varepsilon_k \} ~~ \text{and} ~~\gamma_k:= \beta_k \chi_{A_k}, k = 1, 2, \ldots,
\end{equation*}
we obtain
\begin{equation*}
\| \beta_k - \gamma_k\|_{L_2 \cap L_{\infty}} \leq \max\{ \| \beta_k \chi_{(0, \infty) \setminus A_k} \|_{L_{\infty}}, \| \beta_k\|_{L_2}\} \leq 
\varepsilon_k, k = 1, 2, \ldots .
\end{equation*}
Thus, choosing $\varepsilon_k$ sufficiently small and taking into account inequalities (\ref{1}), by the principle of small perturbations (cf. 
\cite[Theorem 1.3.9]{AK}), we see that the sequences $\{\beta_k\}$ and $\{\gamma_k\}$ are equivalent in $L_2 \cap L_{\infty}$ and 
the subspace $[\gamma_k]$ is complemented (together with $[\beta_k]$) in the latter space.

Denote $A:= \bigcup\limits_{k=1}^{\infty} A_k$. We have $m(A) \leq \sum\limits_{k=1}^{\infty} m(A_k) \leq \sum\limits_{k=1}^{\infty} \varepsilon_k < \infty$ 
and hence the space
\begin{equation*}
(L_2 \cap L_{\infty})(A):= \{x \in L_2 \cap L_{\infty}: \text{supp} \, x \subset A \}
\end{equation*}
coincide with $L_{\infty}(A)$ (with equivalence of norms). As a result, $L_{\infty}(A)$ contains the complemented subspace $[\gamma_k]$, 
which is isomorphic to $l_2$. Since this is a contradiction with \cite[Theorem 5.6.5]{AK}, our initial assumption on the existence of a sequence 
$\{y_{k, i_k}^{n_k}\}_{k=1}^{\infty}$ satisfying (\ref{2}) fails. 

Thus, there are $c > 0$ and $k_0 \in \mathbb N$ such that 
\begin{equation*}
\| y_{k_0, i}^{n}\|_{L_2} \geq c ~~ \text{for all} ~~ n \in \mathbb N ~~ \text{and} ~~ i = 1, 2, \ldots, 2^{k_0}.
\end{equation*}
Then, by the generalized Parallelogram Law (see \cite[Proposition 6.2.9]{AK}), we have
\begin{equation*}
\int_0^1 \| \sum\limits_{i=1}^{2^{k_0}} r_i(s) y_{k_0, i}^{n} \|_{L_2}^2 ds = \sum\limits_{i=1}^{2^{k_0}} \| y_{k_0, i}^{n} \|_{L_2}^2 
\geq c^2 \, 2^{k_0}, n \in \mathbb N,
\end{equation*}
where $r_i = r_i(s)$ are the Rademacher functions. Hence, there exist $\theta_i^n = \pm 1, n = 1, 2, \ldots, i = 1, 2, \ldots, 2^{k_0}$ 
such that $ \| \sum\limits_{i=1}^{2^{k_0}} \theta_i^n y_{k_0, i}^{n} \|_{L_2} \geq c \, 2^{k_0/2}, n \in \mathbb N$, or equivalently 
$ \| \sum\limits_{i=1}^{2^{k_0}} \theta_i^n v_{k_0, i}^{n} \|_{L_2} \geq c, n \in \mathbb N$. So, setting
\begin{equation*}
f_n:= \sum\limits_{i=1}^{2^{k_0}} \theta_i^n u_{k_0, i}^{n}, \, g_n:= \sum\limits_{i=1}^{2^{k_0}} \theta_i^n v_{k_0, i}^{n}, 
\end{equation*}
we have
\begin{equation} \label{3}
\| f_n\|_{L_2 + L_{\infty}} = 1 ~~\text{and} ~~ \| g_n\|_{L_2} \geq c, ~ n = 1, 2, \ldots .
\end{equation}
Moreover, by the definition of the norm in $L_2 + L_{\infty}$ and the fact that
$$
\Big|\sum\limits_{n=1}^m f_n\Big| = \Big|\sum\limits_{n=1}^m \sum\limits_{i=1}^{2^{k_0}} \theta _i^n u_{k_0, i}^n\Big| 
= \sum\limits_{n=1}^m \sum\limits_{i=1}^{2^{k_0}} \chi_{\Delta_{k_0, i}^{n}} = \chi_{(0,m]},
$$
we obtain 
\begin{equation} \label{4}
\| \sum\limits_{n=1}^m f_n\|_{L_2 + L_{\infty}} = \| f_1\|_{L_2} = 1, \, m = 1, 2, \ldots .
\end{equation}
On the other hand, since $\{f_n\}$ is an $1$-unconditional sequence in $L_2 + L_{\infty}$, for each $t \in [0, 1]$ we have
\begin{equation*}
\| \sum\limits_{n=1}^m f_n\|_{L_2 + L_{\infty}}^2 = \| \sum\limits_{n=1}^m r_n(t) \,f_n\|_{L_2 + L_{\infty}}^2 \geq 
\frac{1}{\| T\|^2}  \| \sum\limits_{n=1}^m r_n(t) \,g_n\|_{L_2 \cap L_{\infty}}^2.
\end{equation*}
Integrating this inequality, by the generalized Parallelogram Law and (\ref{3}), we obtain
\begin{eqnarray*}
\| \sum\limits_{n=1}^m f_n\|_{L_2 + L_{\infty}}^2 
&\geq&
\frac{1}{\| T\|^2} \int_0^1 \| \sum\limits_{n=1}^m r_n(t) \, g_n\|_{L_2 \cap L_{\infty}}^2 dt \geq 
\frac{1}{\| T\|^2} \int_0^1 \| \sum\limits_{n=1}^m r_n(t) \, g_n\|_{L_2}^2 dt \\
& =& 
\frac{1}{\| T\|^2}  \sum\limits_{n=1}^m \| \, g_n\|_{L_2}^2 \geq \Big(\frac{c}{\| T\|}\Big)^2 \cdot m, ~ m = 1, 2, \ldots .
\end{eqnarray*}
Since the latter inequality contradicts (\ref{4}), the proof is completed.
\end{proof}

\begin{remark} \label{Rem1} Using the same arguments as in the proof of the above theorem, we can show that the spaces 
$L_p +L_{\infty}$ and $L_p \cap L_{\infty}$ are not isomorphic for every $1 \leq p < \infty$. This gives a new proof of Theorem 1 
from \cite{AM}. However, note that in the latter paper (see  Theorems 3 and 5) it is proved the stronger result, 
saying that the space $L_p \cap L_{\infty}, p \neq 2$, does not contain any complemented subspace isomorphic to $L_p(0, 1)$.
\end{remark}


\section{ $L_p + L_q$ and $L_p \cap L_q$ are not isomorphic for $1 < p, q < \infty$}

Both spaces $L_p +L_q$ and $L_p \cap L_q$ for all $1 \leq p, q \leq \infty$ are special cases of Orlicz spaces on $(0, \infty)$.

A function $M \colon [0,\infty )\rightarrow [0,\infty ]$ is called a \textit{Young function} (or \textit{Orlicz function} if it is
finite-valued) if $M$ is convex, non-decreasing with $M(0)=0$; we assume also that $M$ is neither identically zero 
nor identically infinity on $(0,\infty )$, $\lim_{u\rightarrow 0+} M(u) = M(0) = 0$ and $\lim_{u\rightarrow \infty } M(u) = \infty $.

The \textit{Orlicz space} $L_M = L_M (I)$ with $I = (0, 1)$ or $I = (0, \infty)$ generated by the Young function $M$ is defined as 
$$
L_M (I) 
= \{ x \in L^0 (I): \rho_{M}(x/\lambda) < \infty \ \text{for some} \ \lambda = \lambda (x) > 0\}, 
$$
where $\rho_{M}(x): = \int_I M( |x(t)|) \,d t$. It is a Banach space with the \textit{Luxemburg--Nakano norm}
$$
\|x\|_{L_M} = \inf \{\lambda > 0: \rho_{M}(x/\lambda) \leq 1 \}
$$
and is a symmetric space on $I$ (cf. \cite{BS}, \cite{KR}--\cite{RR}). Special cases on $I = (0, \infty)$ are the following 
(cf. \cite[pp. 98--100]{Ma89}):

(a) For $1 \leq p, q < \infty$ let $M(u) = \max(u^p, u^q)$, then $L_M = L_p \cap L_q$.

(b) For $1 \leq p < \infty$ let 
\begin{equation*}
M(u) = \left\{ 
\begin{array}{cc}
u^p & \text{if } 0 \leq u \leq 1, \text{ }
\\ 
\infty & \text{if} \, 1 < u < \infty,
\end{array}
\right.  \text{then} ~~ L_M = L_p \cap L_{\infty}.
\end{equation*}

(c) For $1 \leq p, q < \infty$ let $M(u) = \min (u^p, u^q)$, then $M$ is not a convex function on $[0, \infty)$, but 
$M_0(u) = \int_0^u \frac{M(t)}{t} dt$ is convex and $M(u/2) \leq M_0(u) \leq M(u)$ for all $u > 0$, which gives 
$L_M = L_{M_0} = L_p + L_q$.

(d) For $1 \leq p < \infty$ let 
\begin{equation*}
M(u) = \left\{ 
\begin{array}{cc}
0 & \text{if } 0 \leq u \leq 1, \text{ }
\\ 
u^p - 1 & \text{if} \, 1 < u < \infty, 
\end{array}
\right. \text{then} ~~ L_M = L_p + L_{\infty}.
\end{equation*}

A Young (Orlicz) function $M$ satisfies the $\Delta_2$-{\it condition} if $0 < M(u) < \infty$ for $u > 0$ and there exists a constant 
$C \geq 1$ such that $M(2u) \leq C M(u)$ for all $u > 0$. An Orlicz space $L_M(0, \infty)$ is separable if and only if 
$M$ satisfies the $\Delta_2$-condition (cf. \cite[pp. 107--110]{KR}, \cite[Thm 4.2 (b)]{Ma89}, \cite[p. 88]{RR}). To each 
Young function $M$ one can associate another convex function $M^*$, i.e., the {\it complementary function} to $M$, which 
is defined by $M^* (v) = \sup_{u>0} \, [uv - M(u)]$ for $v \geq 0$. Then $M^*$ is also a Young function and $M^{**} = M$. 
An Orlicz space $L_M(0, \infty)$ is reflexive if and only if $M$ and $M^*$ satisfy the $\Delta_2$-condition 
(cf. \cite[Thm 9.3]{Ma89}, \cite[p. 112]{RR}).

\begin{theorem} \label{Thm3} Let $M$ and $N$ be two Orlicz functions on $[0, \infty)$ such that both spaces $L_M(0, \infty)$ and 
$L_N(0, \infty)$ are reflexive. Suppose that $L_M(0, \infty)$ and $L_N(0, \infty)$ are isomorphic. Then, the functions $M$ and $N$ 
are equivalent for $u \geq 1$, that is, there are constants $a, b > 0$ such that $a M(u) \leq N(u) \leq b M(u)$ for all $u \geq 1$.
\end{theorem}
\begin{proof} If both functions $M$ and $N$ are equivalent to the function $u^2$ for $u \geq 1$, then nothing has to be proved. So, suppose that 
the function $M$ is not equivalent to $u^2$. Then, clearly, $L_M(0, 1)$ is a complemented subspace of $L_M(0, \infty)$ and $L_M(0, 1)$ 
is different from $L_2(0, 1)$, even up to an equivalent renorming. By hypothesis, $L_N(0, \infty)$ contains a complemented subspace 
isomorphic to $L_M(0, 1)$. Then, by \cite[Corollary 2.e.14(ii)]{LT} (see also \cite[Thm 7.1]{JMST}) $L_M(0, 1) = L_N(0, 1)$ up to equivalent 
norm. This implies that $M$ and $N$ are equivalent for $u \geq 1$ (cf. \cite[Thm 8.1]{KR}, \cite[Thm 3.4]{Ma89}).
\end{proof}

\begin{corollary} \label{Cor1} 
Let $1 < p, q < \infty, p \neq q$, then $(L_p + L_q)(0, \infty)$ and $(L_p \cap L_q)(0, \infty)$ are not isomorphic. 
\end{corollary}
\begin{proof} For such $p, q$ the Orlicz spaces $(L_p + L_q)(0, \infty)$ and $(L_p \cap L_q)(0, \infty)$ are reflexive,  
and are generated by the Orlicz functions $M(u) = \min (u^p, u^q)$ and $N(u) = \max (u^p, u^q)$ respectively, which are 
not equivalent for $u \geq 1$ whenever $p \neq q$. Thus, by Theorem \ref{Thm3}, these spaces cannot be isomorphic.
\end{proof}

\section{ Proof of Theorem 1}

\begin{proof}[Proof of Theorem \ref{Thm1}] 

We consider four cases.

(a) For $p \in [1, 2) \cup (2, \infty)$ and $q = \infty$ it was proved in \cite[Theorem 1]{AM}. 

(b) For $p = 2$ and $q = \infty$ it is proved in Theorem \ref{Thm2}.  

(c) Let $p = 1$ and $1 < q < \infty$. If we assume that $L_1 + L_q$ and $L_1 \cap L_q$ are isomorphic, then the dual spaces 
will be also isomorphic. The dual spaces are $(L_1 + L_q)^* = L_{q^{\prime}} \cap L_{\infty}$ and 
$(L_1 \cap L_q)^* = L_{q^{\prime}} + L_{\infty}$, where $1/q + 1/q^{\prime} = 1$. By (a) and (b), the spaces 
$L_{q^{\prime}} + L_{\infty}$ and $L_{q^{\prime}} \cap L_{\infty}$ are not isomorphic thus its preduals cannot be isomorphic.

(d) For $1 < p, q < \infty, p \neq q$ it follows from Corollary \ref{Cor1}, and the proof is completed.

\end{proof}



\begin{thebibliography}{99}
\small{

\bibitem{AK} F. Albiac, N. J. Kalton, \textit{Topics in Banach Space Theory}, Springer-Verlag, New York 2006.
\vspace{-2mm}

\bibitem{AM} S. V. Astashkin and L. Maligranda, \textit{$L_p + L_{\infty}$ and $L_p \cap L_{\infty}$ are not isomorphic 
for all $1 \leq p < \infty, p \neq 2$}, Proc. Amer. Math. Soc. 146 (2018), no. 5, 2181--2194.
\vspace{-2mm}

\bibitem{BS} C. Bennett and R. Sharpley, {\it Interpolation of Operators}, Academic Press, Boston 1988.
\vspace{-2mm}

\bibitem{BL} J. Bergh and J. L\"ofstr\"om, {\it Interpolation Spaces. An Introduction}, Springer-Verlag, Berlin-New York 1976.
\vspace{-2mm}

\bibitem{DC91} N. L. Carothers and S. J. Dilworth, {\it Some Banach space embeddings of classical function spaces}, Bull. Austral. Math. Soc. 43 (1991), no. 1, 73--77.
\vspace{-2mm}

\bibitem{Di88} S. J. Dilworth, {\it Intersection of Lebesgue spaces $L_1$ and $L_2$}, Proc. Amer. Math. Soc. 103 (1988), no. 4, 
1185--1188.
\vspace{-2mm}

\bibitem{Di90} S. J. Dilworth, {\it A scale of linear spaces related to the $L_p$ scale}, Illinois J. Math. 34 (1990), no. 1, 140--158.
\vspace{-2mm}

\bibitem{Ho70} T. Holmstedt, \textit{Interpolation of quasi-normed spaces}, Math. Scand. 26 (1970), 177--199.
\vspace{-2mm}

\bibitem{JMST} W. B. Johnson, B. Maurey, G. Schechtman and L. Tzafriri, \textit{Symmetric structures in Banach spaces}, Mem. Amer. 
Math. Soc. 19 (1979), no. 217, v+298 pp. 
\vspace{-2mm}

\bibitem{KR} M. A. Krasnosel'ski{\u \i} and Ja. B. Ruticki{\u \i}, \textit{Convex Functions and Orlicz Spaces}, Noordhoff, Groningen 1961. 
\vspace{-2mm}

\bibitem{KPS} S. G. Krein, Yu. I. Petunin and E. M. Semenov, \textit{Interpolation of Linear Operators}, Amer. Math. Soc., Providence 1982.
\vspace{-2mm}

\bibitem{LT} J. Lindenstrauss and L. Tzafriri, \textit{Classical Banach Spaces, II. Function Spaces}, Springer-Verlag, Berlin--New 
York 1979.
\vspace{-2mm}

\bibitem{Ma84} L. Maligranda, {\it The K-functional for symmetric spaces}, Lecture Notes in Math. 1070 (1984), 169--182.
\vspace{-2mm}

\bibitem{Ma89} L. Maligranda, \textit{Orlicz Spaces and Interpolation},  Seminars in Math. 5, University of Campinas, Campinas 
1989.
\vspace{-2mm}

\bibitem{RR} M. M. Rao and Z. D. Ren, \textit{Theory of Orlicz Spaces}, Marcel Dekker, New York 1991.

}
\end{thebibliography}
\end{document}